\documentclass[11pt]{article}
\usepackage{latexsym}
\usepackage{epsfig,enumerate}
\usepackage{amsmath,amsthm,amssymb}
\usepackage{algorithm2e}
\usepackage{appendix}
\usepackage{hyperref}
\usepackage{comment}
\usepackage{float}

\usepackage[usenames]{color}\usepackage{graphicx}
\numberwithin{equation}{section}
\definecolor{brown}{cmyk}{0, 0.72, 1, 0.45}
\definecolor{grey}{gray}{0.5}

\renewcommand{\epsilon}{\varepsilon}

\def\s{\sigma}

\def\t{\tau}

\newcounter{rot}

\def\nn{\nonumber}
\def\a{\alpha} \def\b{\beta} \def\d{\delta} \def\D{\Delta}
    \def\g{\gamma}
  \def\k{\kappa} 
     \def\l{\ell}
   
  \def\s{\sigma} 
\def\t{\tau}

\newtheorem{conjecture}{Conjecture}
\newtheorem{opq}{Open Question}
\newtheorem*{conjecture*}{Conjecture}
\newtheorem{theorem}{Theorem}[section]
\newtheorem*{theorem*}{Theorem}

\newtheorem{claim}[theorem]{Claim}

\newtheorem{corollary}[theorem]{Corollary}

\newtheorem{remark}[theorem]{Remark}

\newcommand{\rbrac}[1]{\left(#1\right)}
\newcommand{\cbrac}[1]{\left\{#1\right\}}
\newcommand{\sbrac}[1]{\left[{ #1}\right]}

\newcommand{\bfrac}[2]{\left(\frac{#1}{#2}\right)}

\newcommand{\toloc}{\stackrel{\textrm{loc}}{\to}}

\newcommand{\set}[1]{\left\{#1\right\}}

\def\E{\mathbb{E}}

\def\P{\mathbb{P}}
\def\Pr{\mathbb{P}}

\newcommand{\eps}{\varepsilon}

\newcommand{\of}[1]{\left( #1 \right) }

\newcommand{\sqbs}[1]{\left[ #1 \right]}

\newcommand{\Mean}[1]{\E\sqbs{#1}}
\newcommand{\tbf}[1]{\textbf{#1}}
\renewcommand{\Pr}[1]{\mathbb{P}\left[ #1 \right]}

\newcommand{\ignore}[1]{}

\newcommand{\beq}[1]{\begin{equation}\label{#1}}
\newcommand{\eeq}{\end{equation}}

\newcommand{\mc}[1]{\mathcal{#1}}

\def\cH{{\cal H}}

\title{The Matching Process and Independent Process in Random Regular Graphs and Hypergraphs}

\author{Deepak Bal\thanks{Department of Mathematics, Montclair State University, Montclair, NJ, 07043 U.S.A. \texttt{deepak.bal@montclair.edu}}\and  Patrick Bennett\thanks{Department of Mathematics, Western Michigan University, Kalamazoo, MI, 49008 U.S.A.
\texttt{patrick.bennett@wmich.edu}}}
\date{}
\begin{document}
\maketitle

\begin{abstract}
In this note, we analyze two random greedy processes on sparse random graphs and hypergraphs with a given degree sequence. First we analyze the matching process, which builds a set of disjoint edges one edge at a time; then we analyze the independent process, which builds an independent set of vertices one vertex at a time.  We use the differential equations method and apply a general theorem of Warnke. Our main contribution is to significantly reduce the associated systems of differential equations and simplify the expression for the final size of the matching or independent set.  
\end{abstract}

\section{Introduction}
 A \tbf{matching} in a hypergraph is a collection of vertex-disjoint edges. The algorithmic theory of matchings in graphs is very well studied. In particular, Edmond's Blossom Algorithm provides a polynomial time algorithm to find the largest size matching in a general graph (see e.g. \cite{LP}). However,  for $r\ge 3$ the problem of determining whether a given matching is maximum in a $r$-uniform hypergraph is NP-hard. An \tbf{independent set} is a set of vertices containing no edge. The problem of determining whether a given independent set is maximum in a $r$-uniform hypergraph is NP-hard for all $r\ge 2$.

 In this paper we study the natural random greedy algorithms for producing matchings and independent sets in hypergraphs. For both matchings and independent sets, we are looking for a collection of objects that do not ``conflict" with each other. The natural random greedy algorithm is then to build our collection of objects one by one, at each step choosing a random object that does not conflict with previous choices, until no such choice is possible and we have a maximal collection. We will call the random greedy matching algorithm the \tbf{ matching process}, and we will call the random greedy independent set algorithm the \tbf{ independent process}. See Algorithms \ref{algo:match} and \ref{algo:ind} for details.

\begin{minipage}{.4\textwidth}
\begin{algorithm}[H]\label{algo:match}
  \SetKwInOut{Input}{Input}
   \SetKwInOut{Output}{Output}
	\caption{\textsc{matching process}}
	\Input{Hypergraph $H=(V,E)$}
	\Output{Matching $M$}
	$M = \emptyset$\;
	\While{$E\neq \emptyset$}
	{Select $e \in E$ uniformly at random\;
	 $M	\gets M \cup \set{e}$\;
	 $E\gets E\setminus \set{e'\in E \,:\, e'\cap e \neq \emptyset}$\;
	}
	\Return $M$\;
	
\end{algorithm} 
\end{minipage}
\hfill
\begin{minipage}{.5\textwidth}
\begin{algorithm}[H]\label{algo:ind}
  \SetKwInOut{Input}{Input}
   \SetKwInOut{Output}{Output}
	\caption{\textsc{independent process}}
	\Input{Hypergraph $H=(V,E)$}
	\Output{Independent set $I$}
	$I = \emptyset$\;
	\While{$V \neq \emptyset$}
	{Select $v \in V$ uniformly at random\;
	 $I	\gets I \cup \set{v}$\;
	 $V \gets V\setminus \set{v'\in V \,:\, I \cup \{v'\} \mbox{ contains an edge of } E}$\;
	}
	\Return $I$\;
	
\end{algorithm} 
\end{minipage}

\subsection{The matching process}

 The performance of the matching process was analyzed on arbitrary (ordinary) graphs by Dyer and Frieze \cite{DF91}, on dense random graphs by Tinhofer \cite{Tin}, and on sparse random graphs by Dyer, Frieze and Pittel \cite{DFP93}. The analysis in \cite{DFP93} was extended to sparse random hypergraphs in the Ph.D. thesis of Chebolu \cite{Che}. The algorithm was analyzed on deterministic hypergraphs by Wormald \cite{Wor99}, who proved that for a $r$-uniform, $\D$-regular hypegraph on $n$ vertices where $\D \rightarrow \infty$ sufficiently fast as $n\rightarrow \infty$, with high probability\footnote{We say a sequence of events $A_n$ happens \tbf{with high probability} (whp), if $\lim_{n\to\infty}\Pr{A_n} =1$ } the matching process outputs a matching that covers all but $o(n)$ vertices. Bennett and Bohman \cite{BeBoh} improved Wormald's bound on the number of uncovered vertices assuming the hypergraph satisfies a co-degree condition. However there is no reason to believe that the analysis in \cite{BeBoh} is optimal, and a folklore conjecture claims that there should be a better bound:
 \begin{conjecture}[Folklore]\label{conj:folk}
 For an $r$-uniform, $\D$-regular hypergraph with $\D \rightarrow \infty$ and assuming some weak co-degree conditions, the fraction of unmatched vertices at the end of the matching process should be \begin{equation}\label{eq:conj}
\D^{-\frac{1}{r-1} + o(1)}.
\end{equation}
 \end{conjecture} \noindent
See \cite{BeBoh} for more discussion on this conjecture.


 We now describe our random hypergraph model. Suppose $\D=\D(n)$ is a function of $n$ and 
\[\tbf{n} = \tbf{n}(n) = (n_1, n_2. \ldots, n_\D) \in \mathbb{N}^{\D}\] is a vector such that $n_\D$ is positive\footnote{ we use the convention $0 \in \mathbb{N}$} and $\sum_{i=1}^{\D}n_i = n$.
Let $\cH(n, r, \tbf{n})$ represent the probability space with uniform distribution over all $r$-uniform hypergraphs on $n$ vertices with $n_i$ many vertices of degree $i$. In the special case when $n_\D =n$, $\cH(n,r,\tbf{n})$ represents the random $r$-uniform, $\D$-regular hypergraph and we denote this by $\cH(n,r,\D)$.  

Bohman and Frieze \cite{BF11} studied matchings in the {\bf fixed degree sequence random graph model}, which in our terminology is the graph case $\cH(n,2,\tbf{n})$ where the vector $\frac 1n \tbf{n}(n)$ is constant (i.e. the same for all $n$). In particular they analyzed the performance of the Karp-Sipser matching algorithm and determined a sufficient condition on the degree sequence in order for the random graph to have an almost perfect matching (i.e. a matching that leaves $o(n)$ vertices unmatched) w.h.p.. 

Our main result for the matching process is the following. From now on all asymptotics are as $n \rightarrow \infty$.
 \begin{theorem} \label{thm:match}
Suppose $\D = o(n^{1/3})$. Let 
\begin{equation}\label{eqn:d1d2}
    d_1:= \frac1n \sum_{i=1}^\D in_i, \qquad d_2:= \frac1n \sum_{i=1}^\D i(i-1)n_i,
\end{equation}

and suppose that $d_1, d_2 = \Theta(1)$. Let 
\begin{equation} \label{eqn:fgdef}
  f(x) := \frac1n \sum_{i=1}^\D n_i x^i, \qquad g(x):= \frac{1}{d_1^{r-1}} - \int_x^1 \frac{dz}{(f'(z))^{r-1}} . 
\end{equation}
 Then with high probability  the matching process run on  $H\sim \cH(n,r,\tbf{n})$ terminates with a matching that covers all but 
\[
f(g^{-1}(0)) n + o(n)
\]
vertices. 
\end{theorem}
 \noindent 
 Note that $f$ is the generating function for the sequence $(n_i/n)$. Note that in the fixed degree sequence case (i.e. when the vector $\frac 1n \tbf{n}(n)$ does not depend on $n$) we have that $f(x)$ is just a fixed polynomial. In that case the antiderivative in the definition of $g(x)$ can (at least in principle) be calculated using partial fractions. Unfortunately, in general this antiderivative will be messy and involve logarithms and arctangents, in which case one would probably resort to numerical methods to approximate $g^{-1}(0)$. However, the solution can be written explicitly in the case corresponding to sparse regular hypergraphs:
\begin{corollary}\label{cor:main}
Suppose that $r, \D \ge2$ are fixed   such that $r+\Delta \ge 5$. Then the matching process run on $H\sim \cH(n,r, \D)$ produces a matching which covers all but \[ \bfrac{1}{(r-1)(\D-1)}^\frac{\D}{(r-1)(\D-1)-1}n + o(n)\] many vertices whp.
\end{corollary}

Actually, we will also extend the above result to deterministic high-girth hypergraphs, using a result of Krivelevich, M\'esz\'aros, Michaeli and Shikhelman \cite{KMMS} (see subsection \ref{sec:KMMS} for a proof sketch). A Berge $k$-cycle is a sequence of $k$ distinct vertices $v_1, \ldots, v_k$ and $k$ distinct edges $e_1, \ldots, e_k$ such that $e_i$ contains $v_i$ and $v_{i+1}$ (indices modulo $k$).
The \emph{(Berge) girth} of a hypergraph $H$ is the smallest integer $k$ such that $H$ contains a Berge $k$-cycle (we say the girth is infinity if no Berge cycle exists).

\begin{corollary}\label{cor:matchgirth}
For fixed $r, \D$ such that $r+\Delta \ge 5$, let $\mc{H}_n$ be a sequence of $r$-uniform $\D$-regular hypergraphs with girth tending to infinity. Then $\mc{H}_n$ has a matching that covers all but at most 
\[ \bfrac{1}{(r-1)(\D-1)}^\frac{\D}{(r-1)(\D-1)-1}n + o(n)\]
vertices. Moreover, the matching process w.h.p. returns such a matching. 
\end{corollary}

The main term in the above corollaries lends some credence to Conjecture \ref{conj:folk}. Indeed, if $\Delta \rightarrow \infty$  then 
\begin{equation*}
\bfrac{1}{(r-1)(\D-1)}^\frac{\D}{(r-1)(\D-1)-1} = \D^{-\frac{1}{r-1} + o(1)}
\end{equation*}
which is the same conjectured fraction from line \eqref{eq:conj}.

In \cite{CFMR96}, Cooper, Frieze, Molloy and Reed used the small subgraph conditioning method of Robinson and Wormald \cite{RW92, RW94} to prove that for $r \ge 3, \D \ge 2$
\begin{equation}\label{eqn:CFMR}
    \lim_{n\to\infty}\P\sqbs{\cH(n,r,\D) \textrm{ has a perfect matching}} = 
\begin{cases}
1 \quad \textrm{if $r < \sigma_\D$} \\
0 \quad \textrm{if $r > \sigma_\D$}
\end{cases}
\end{equation}

where $\s_\D := \frac{\log \D }{(\D -1)\log\bfrac{\D}{\D-1}} +1$. Thus for any $r$, $f(r) = \min\set{\D:\, r < \s_\D }$ gives the threshold of $\D$ such that $\cH(n, r, \D)$ has a perfect matching and for large $r$, $f(r)\sim e^{r-1}.$ Thus it is interesting to note that near this threshold, the greedy algorithm finds a matching which, asymptotically in $r$, covers only $(1- e^{-1} + o(1))$ fraction of the vertices even though there is a perfect matching w.h.p..

\subsection{The independent process}
The independent process, also referred to as the random greedy independent set algorithm (see e.g. \cite{BeBoh}) or random sequential adsorption in the realms of chemistry and physics (see e.g. \cite{Pen01}), has been studied for many years. The algorithm was studied on the binomial random graph (in the context of coloring) by Grimmett and McDiarmid \cite{GM75}, and Bollob\'{a}s and Erd\H{o}s \cite{BE76}. The algorithm was analyzed on random regular graphs by Wormald (see \cite{Wor95, Wor97}) and on the fixed degree sequence graph model by Brightwell, Janson and {\L}uczak \cite{BJL}. 

Our main result for the independent process is as follows.

\begin{theorem}\label{thm:ind}
Suppose $\D =o\rbrac{n^{1/4} / \log^{1/2}n}$ and that $d_1, d_2 = \Theta(1)$ (see line \eqref{eqn:d1d2}). Let $f(x)$ be as defined in line \eqref{eqn:fgdef}. Let $\a=\a(x), \b=\b(x)$ be the unique solution to the system
\begin{equation}\label{eqn:alphadiffeq}
   \a' = \frac{f'(1-\a^{r-1})}{d_1 f(1-\a^{r-1})}, \quad \b' = -\frac{1}{f(1-\a^{r-1})}, \quad  \a(0)=0, \quad \b(0)=1.
\end{equation}
Then with high probability  the independent process run on $H\sim \cH(n,r, \tbf{n})$ terminates with an independent set of size 
\begin{equation}\label{eqn:finalindsize}
   \b^{-1}(0) n + o(n)
\end{equation}
vertices. 
\end{theorem}

Note that since rational functions have elementary antiderivatives, in the fixed degree sequence case (when $f$ is a polynomial) we have that $\a$ is the inverse of an elementary function. Also we have that 
\begin{align}\label{eqn:gammadef}
    \b & = 1- \int_0^t  \frac{d\t}{f(1-\a(\t)^{r-1})}\\
    &= 1- \int_{\a(0)}^{\a(t)}  \frac{d_1 du}{f'(1-u^{r-1})} \\
    &=\g(\a(t))
\end{align}
where the second line follows from the substitution $u=\a(\t)$ and
\begin{equation}
    \g(x):= 1- \int_{0}^{x} \frac{d_1 du}{f'(1-u^{r-1})}
\end{equation}
is an elementary function. So roughly speaking, both $\a, \b$ are ``close" to being elementary.

The proofs of Theorems \ref{thm:match} and \ref{thm:ind}  are applications of the so-called differential equations method\footnote{Readers unfamiliar with the differential equations method should refer to the surveys \cite{DM10} and \cite{Wor99} or more specifically Theorem \ref{thm:Lutz} in Section \ref{sec:conc} of this paper.}. In each proof we will directly apply a general theorem of Warnke \cite{Lutz} to show concentration a family of random variables that evolve with the process. In Section \ref{sec:setup} we will set up the analysis of each process, and for each process we will derive a system of differential equations describing how our variables evolve with the process.  In Section \ref{sec:conc}, we will complete the proof of Theorems \ref{thm:match} and \ref{thm:ind} by applying Warnke's result. In section \ref{sec:UBs}, we will use the first moment method to provide some bounds on the largest size of a matching or independent set in a random regular hypergraph.

\subsection{Connections to other results on the independent process}

Theorem \ref{thm:ind} has connections to two other recent results. Brightwell, Janson and {\L}uczak \cite{BJL} analyzed the independent process  on the graph case $\cH(n,2, \tbf{n})$ and determined the size of the final independent set. Thus our Theorem \ref{thm:ind} is (roughly speaking, ignoring a few technical assumptions) a generalization of \cite{BJL} and we expect that our formula for the final size of the independent set should be the same as the one in \cite{BJL} when $r=2$. In subsection \ref{sec:BJL} we show this is the case. Pippenger \cite{Pip} and independently Lauer and Wormald \cite{LW} analyzed the independent process on (deterministic) regular graphs of high girth and found the expected size of the final independent set (Pippenger \cite{Pip} also did something similar for the matching process). Gamarnik and Goldberg \cite{GG10} then established concentration of the final size of the independent set or matching produced by the processes (again working on regular graphs of high girth). Nie and Verstra\"ete \cite{NV} also analyzed the independent process on regular hypergraphs of large girth. Since sparse random regular hypergraphs have very few short cycles, it is natural to expect that the final independent set obtained in the setting of \cite{NV} should be asymptotically the same as ours. Indeed, in subsection \ref{sec:NV} we will see that this is the case as well. 

\subsubsection{Graph case: Brightwell, Janson and {\L}uczak}\label{sec:BJL}

Brightwell, Janson and {\L}uczak proved the following.
\begin{theorem}[Brightwell, Janson, {\L}uczak \cite{BJL}]
Let $(p_k)_0^\infty$ be a probability distribution, and let $d_1 = \sum_{k=1}^\infty kp_k \in (0, \infty).$ Assume $n_k/n \rightarrow p_k$ for each $k$ and that $\sum_{k=1}^\infty kn_k/n \rightarrow d_1$ as $n \rightarrow \infty$. Assume further that $\sum_{k=1}^\infty k^2 n_k = O(n)$. 

Let $\tau_\infty \in(0, \infty]$ be the unique value such that 
\begin{equation}\label{eqn:tinfdef}
    d_1 \int_0^{\tau_\infty} \frac{e^{-2\s}}{\sum_k kp_k e^{-k\s} }d\s =1.
    \end{equation}
    Then w.h.p. the final size of the independent set produced by the process run on the random graph $\mc{H}(n, 2, \tbf{n})$ is 
    \begin{equation}\label{eqn:BJLresult}
       \rbrac{ d_1 \int_0^{\tau_\infty} \frac{e^{-2\s \sum_kp_k e^{-k\s}}}{\sum_k kp_k e^{-k\s} }d\s} n +o(n)
    \end{equation}
\end{theorem}
Of course the final size of the independent set on line \eqref{eqn:BJLresult} is asymptotically the same as that on line \eqref{eqn:finalindsize}. The forms these expressions take is an artifact of the methods used to analyze the processes. It seems rather unsatisfying to say the expressions are obviously equal because both theorems are true, so we will provide another calculation justifying that they are equal, at least in the fixed degree sequence case (so we do not have to worry about the convergence of $f$ as $n$ grows). Then note that we have $f(e^{-\s}) = \sum_k p_k e^{-k\s}  $ and similarly $\sum_k kp_k e^{-k\s}  = e^{-\s} f'(e^{-\s})$. Thus, the definition of $\tau_\infty$ in line \eqref{eqn:tinfdef} can be rewritten as
\[
d_1 \int_0^{\tau_\infty} \frac{e^{-\s}}{f'(e^{-\s}) }d\s =1
\]
and then, subsituting $u = 1-e^{-\s}$ and using the fact that $\g'(x) = -\frac{1}{f'(1-x)}$ we have
\[
d_1 \int_0^{1-e^{-\tau_\infty}} \frac{1}{f'(1-u) }du = -\g(1-e^{-\tau_\infty}) + \g(0) = 1.
\]
Since $\g(0)=1$ this implies that $\g(1-e^{-\tau_\infty})=0$ which can serve as an equivalent definition of $\t_\infty$. Meanwhile, the coefficient of $n$ in line \eqref{eqn:BJLresult} is equal to 
\[
d_1 \int_0^{1-e^{-\tau_\infty}} \frac{f(1-u)}{f'(1-u) }du.
\]
Substituting $u = \a(w)$ we get 
\[
\int_0^{\a^{-1}(1-e^{-\tau_\infty)}} dw = \a^{-1}(1-e^{-\tau_\infty}) = \a^{-1}(\g^{-1}(0)) = \b^{-1}(0),
\]
which is the coefficient of $n$ on line \eqref{eqn:finalindsize}.

\subsubsection{Regular case: Nie and Verstra\"ete} \label{sec:NV}

Here we relate our result to that of Nie and Verstra\"ete \cite{NV}. Let $H$ be a (deterministic) $\D$-regular $r$-uniform hypergraph with $n$ vertices. Let 
\begin{equation}\label{eqn:hdef}
    h(x) := 1- \sum_{n\ge 0}\binom{n+\D-2}{\D-2}\frac{x^{(r-1)n+1}}{(r-1)n+1}.
\end{equation}
Let
\begin{equation}\label{eqn:fdef}
f(\D, r):= \int_0^1 (1-h^{-1}(x)^{r-1})^\D\,dx.
\end{equation}
The main result in \cite{NV} is as follows.
\begin{theorem}[Nie, Verstra\"ete] \label{thm:NV}
There exists a function $\eps(\D, r, g)$ such that for any fixed $\D, r$ we have that $\eps(\D, r, g)\rightarrow 0$ as $g \rightarrow \infty$ with the following property. For $r, \D \ge 2$, $g \ge 4$, the expected size of the final independent set produced by the independent process is in the interval
\[
[(f(\D, r)- \eps)n, (f(\D, r)+ \eps )n]
\]
where $\eps = \eps(\D, r, g).$
\end{theorem}
Since random regular hypergraphs typically have few short cycles, we expect them to be similar to large girth regular hypergraphs. In particular, we expect that in the regular case $f(\D, r)$ should be equal to $\b^{-1}(0)$ (the coefficient of $n$ in \eqref{eqn:finalindsize}). In the rest of this subsection we justify that they are indeed equal. 

First we would like to ``simplify" $f(\D, r)$ somewhat. Using \eqref{eqn:hdef} and Newton's binomial formula we have
\begin{equation}\label{eqn:hprime}
    h'(x) =- \sum_{n\ge 0}\binom{n+\D-2}{\D-2}x^{(r-1)n} =- \frac{1}{(1-x^{r-1})^{\D-1}}.
\end{equation}

Now we turn to the definition of $f(\D, r)$ in \eqref{eqn:fdef} and make the substitution $u = h^{-1}(x)$. Then 
\[
du  = \frac{dx}{h'(u)} = -(1-u^{r-1})^{\D-1} dx.
\]
So we have
\begin{align}
    f(\D, r) &= \int_0^1 (1-h^{-1}(x)^{r-1})^\D\,dx = \int_{h^{-1}(0)}^{h^{-1}(1 )} -(1-u^{r-1})\; du \nonumber\\
    &= \int_{h^{-1}(0)}^{0} u^{r-1}-1\; du = h^{-1}(0) - \frac1r h^{-1}(0)^r.\label{eqn:fsimp}
\end{align}
It remains to show that the above is equal to $\b^{-1}(0)$.

We turn to the system \eqref{eqn:alphadiffeq} in the $\D$-regular case. We have $d_1=\D$ and $f(x) = x^\D$, so
\[
\a' = \frac{f'(1-\a^{r-1})}{d_1 f(1-\a^{r-1})} = \frac{\D(1-\a^{r-1})^{\D-1}}{\D (1-\a^{r-1})^\D} = \frac{1}{1-\a^{r-1}}.
\]
The above separable differential equation can be solved implicitly (using the initial condition $\a(0)=0$) to obtain 
\[
\a-\frac1r \a^r = t,
\]
or equivalently
\[
\a^{-1}(x) = x - \frac 1r x^r.
\]
Now note that $h$ and $\g$ are the same function. Indeed, we know that $h(0)=\g(0)=1$, and by equations \eqref{eqn:gammadef} and \eqref{eqn:hprime} we have that 
\[
\g'(x) =- \frac{d_1 }{f'(1-x^{r-1})} = - \frac{\D }{\D(1-x^{r-1})^{\D-1}} = h'(x).
\]
Now
\[
\b^{-1}(0) = \a^{-1}(\g^{-1}(0)) = \a^{-1}(h^{-1}(0)) = h^{-1}(0) - \frac1r h^{-1}(0)^r = f(\D, r)
\]
as desired. 

\subsubsection{Krivelevich, M\'esz\'aros, Michaeli and Shikhelman} \label{sec:KMMS}

Krivelevich, M\'esz\'aros, Michaeli and Shikhelman \cite{KMMS} gave a very general analysis of the independent process, showing in many cases (sequences of graphs) of interest  that the final size of the independent set can be approximated using an appropriate ``limiting" object (locally finite graph). While there is no formal statement or proof of a hypergraph analog of the result in \cite{KMMS}, the authors do show that, assuming such an analog is true, it implies something very similar to Nie and Verstra\"ete's result in \cite{NV}. Indeed, as noted in \cite{KMMS}, the local limit (see \cite{KMMS} for technical definitions) of a sequence of $r$-uniform $\D$-regular hypergraphs with girth tending to infinity is $\mathbb{T}_\D^r$, the infinite rooted $r$-uniform $\D$-regular loose tree. Informally this means that if we choose a  vertex $v$ in an $r$-uniform $\D$-regular hypergraph of girth say $3k$, the ball of radius $k$ around $v$ is isomorphic to the ball of radius $k$ around a vertex of $\mathbb{T}_\D^r$. Using the limiting object $\mathbb{T}_\D^r$ to set up a differential equation recovers the result of \cite{NV} (ignoring the precise error bound in the result). 

Similarly, we can obtain a result for matchings on $r$-uniform $\D$-regular hypergraphs of high girth using our Theorem \ref{thm:match} together with the result from Krivelevich, M\'esz\'aros, Michaeli and Shikhelman \cite{KMMS}. Indeed, for any hypergraph $\mc{H}$ we define the line graph $L(\mc{H})$ to have vertex set $E(\mc{H})$ where $e_1$ is adjacent to $e_2$ in $L(\mc{H})$ if and only if $e_1 \cap e_2 \neq \emptyset$. Thus a matching in $\mc{H}$ is precisely an independent set in $L(\mc{H})$. In fact, the matching process run on $\mc{H}$ is precisely the independent process on $L(\mc{H})$.

All of the applications given in \cite{KMMS} are to locally treelike graphs (i.e. graphs whose local limit is a tree). However $L(\mc{H})$ is typically not locally treelike (for example any vertex of degree $d$ in $\mc{H}$ gives us a $d$-clique in $L(\mc{H})$). Thus we cannot use Theorem 1.3 in \cite{KMMS} to analyze the independent process on $L(\mc{H})$. However, we can still use their Theorem 1.2 together with our own analysis to obtain a result. Let $\mc{H}_n$ be a sequence of $r$-uniform $\D$-regular hypergraphs with girth tending to infinity. It is trivial to check that, using the terminology from \cite{KMMS}, that $L(\mc{H}_n) \toloc (L(\mathbb{T}_\D^r), \rho)$ where $\rho$ is an arbitrary vertex of $L(\mathbb{T}_\D^r)$.
It is also not hard to see that $L(\mc{H}(n, r, \D)) \toloc (L(\mathbb{T}_\D^r), \rho)$. Thus, by Theorem 1.2 in \cite{KMMS}, w.h.p. the size of the final independent set when the process is run on $L(\mc{H}_n)$ is asymptotically the same as when the process is run on $L(\mc{H}(n, r, \D))$. In other words, w.h.p. the final matching produced by the matching process is asymptotically the same when run on $\mc{H}_n$ and $\mc{H}(n, r, \D).$ Now using our Theorem \ref{thm:match} we have Corollary \ref{cor:matchgirth}.

\section{Setting up the analysis for the processes} \label{sec:setup}

We generate our random hypergraph $\cH(n, r, \tbf{n})$ as follows. Suppose the degree sum is $n \sum_i i n_i = rm$ (this serves as a definition of $m$). We will have two sets of \tbf{points}: a set $A$ which is the union of $m$ disjoint sets of $r$ points each (these sets will be called the \tbf{edges}), and a set $B$ which is the union of disjoint sets of points (called \tbf{vertices}), where for each $0 \le i \le \D$ there are $n_i n$ vertices containing $i$ points. Note that $|A|=|B|=rm.$ We generate a uniform random \tbf{pairing} (a partition of $A \cup B$ into sets of size 2 each containing one element from $A$ and one from $B$).  We interpret the pairing as a hypergraph as follows: a vertex $v$ is contained in an edge $e$ if and only if there is a point in vertex $v$ that is paired with a point in the edge $e$.

We would like to restrict our attention to hypergraphs without ``loops" (i.e. edges that contain a vertex multiple times) or ``multi-edges" (i.e. two edges consisting of the same set of vertices). We call such hypergraphs \tbf{simple}. The main theorem in Blinovsky and Greenhill's paper \cite{BG16} can help us here. Let
\[
M_0=\sum_{i=1}^{\Delta} n_i,\quad M_1=\sum_{i=1}^\Delta in_i, \quad M_2 = \sum_{i=1}^\Delta i(i-1)n_i.
\]
The main result of \cite{BG16} implies that if $n, M_1 \rightarrow \infty$, $M_2 = O(M_1)$ and $\D = o\rbrac{M_1^{1/3}}$ then 
\[
P[\cH(n, r, \tbf{n}) \mbox{ is simple}] = \Omega(1)
\]

Note that we are not using use the most natural extension to hypergraphs of the standard configuration model (see e.g. \cite{BG16, Bol01, JLR00, Wor97}). In the standard configuration model we have only the set of points $B$ (there is no set $A$), which we randomly partition into sets of $r$ points each. That partition determines the edges of the hypergraph. But a moment's thought reveals that this model is equivalent to the one we use in the present paper (indeed, the neighborhoods of the edges in our set of points $A$ give us a uniform partition of $B$ into parts of $r$ points each). 
The reason for our choice of model is that we would like to analyze the processes step-by-step by revealing only a small part of the pairing at each step (often called the method of deferred decisions). In particular we would like to have the ability to reveal only part of an edge without revealing the entire edge (i.e. we reveal some but not all of the vertices in that edge). This is because in the independent process we are allowed to choose several vertices all in the same edge, so the method of deferred decisions guides us to reveal only what we need to know at the time, i.e. which edges the chosen vertex $v$ is in. It is not necessary to reveal an entire edge unless that edge has $r-1$ vertices in $I$ (in which case the last vertex in that edge must be removed from $V$). Our choice of model does not seem to have any advantage when analyzing the matching process, but for unity we use it there too since it is no more difficult. 

Each process (recall Algorithms \ref{algo:match} and \ref{algo:ind}) has a while loop, and we will call one iteration of the while loop a \tbf{step}, so in the matching process a step means choosing one edge to go into $M$, and in the independent process a step means choosing a vertex to go into $I$. We will break each step into several \tbf{pairings}, where in this context a pairing means revealing the partner of one point (we say that both points are paired after this).

\subsection{Matching process}
Our set up for the matching process is similar to that of Bohman and Frieze in \cite{BF11} where the Karp-Sipser algorithm is analyzed on random graphs with fixed degree sequence. To execute one iteration of the while loop in Algorithm \ref{algo:match}, we choose an edge $e \in E$ and then update $E$ by removing any edge $e'$ that intersects $e$. When we analyze the algorithm we will pair the points that are necessary in order to update $E$, plus a little bit more to aid the analysis. Recall that $E$ is a set of edges, and formally edges here are parts of the partition of $A$ (each of which consists of $r$ points). So we choose one random such edge $e$ and then update $E$ as follows. First we pair all the points in $e$ to see which vertices are incident with $e$. Then we pair all the points in those vertices since the partners of these points are in edges $e'$ that must be removed from $E$. We then pair all the points in any such edge $e'$ being removed from $E$, which will help us track random variables to analyze the process. 


At step $j$, say there are $Y_i = Y_i(j)$ many vertices with exactly $i$ unpaired points for $i=1, \ldots, \D$. In a slight abuse of notation we will let $E=E(j)$ be the number of edges in the set $E$ at step $j$ (we will often use the same notation for a set and its cardinality). 
Let 
\[
M_0=\sum_{i=1}^{\Delta} Y_i,\quad M_1=\sum_{i=1}^\Delta iY_i, \quad M_2 = \sum_{i=1}^\Delta i(i-1)Y_i.
\]
Then since the number of unpaired points in $A$ is always the same as in $B$ we have
\[r E(j) = \sum_{i = 1}^{\D}i Y_i(j) = M_1.\] 
We will now start to set up our application of Warnke's Theorem (see Theorem \ref{thm:Lutz} for a preview). We will write a system of differential equations intended to model the change in each of our variables over one step. We will assume that we are at a step $j$ where $M_1$ is still linear, i.e. $M_1 = \Omega(n)$.  For $i=1,\ldots, \D$, we have that the expected change in $Y_i$ over one step (conditional on the current values of the $Y_i$) is

\begin{align}
&\E\sqbs{Y_i(j+1)- Y_i(j)| \tbf{Y}(j)} \nonumber\\
&\qquad \qquad \qquad \qquad =-r \cdot \frac{iY_i}{M_1} +r\sum_{k=1}^{\D}\frac{k Y_k}{M_1}(k-1)(r-1)\of{ \frac{(i+1)Y_{i+1}}{M_1}-\frac{iY_i}{M_1} } + O\bfrac{\D^2}{n}\nonumber\\
& \qquad \qquad \qquad \qquad =-r \cdot \frac{iY_i}{M_1} +r(r-1)\frac{M_2}{M_1}\of{\frac{(i+1)Y_{i+1}}{M_1}-\frac{iY_i}{M_1} } + O\bfrac{\D^2}{n}\nonumber\\
& \qquad \qquad \qquad \qquad =  \rbrac{\frac{(i+1)r(r-1)M_2}{M_1^2}}Y_{i+1} - \rbrac{\frac{ir}{M_1} + \frac{ir(r-1)M_2}{M_1^2}}Y_{i} + O\bfrac{\D^2}{n}.\label{eq:mExpOneStep}
\end{align}
Indeed, we account for the terms on the second line as follows. When we choose our edge $e$ to go into $M$, we pair the $r$ points in $e$, each of which has a probability $\frac{i Y_i +O(\D)}{M_1 +O(\D)}= \frac{i Y_i }{M_1 } + O\bfrac{\D}{n}$ of being paired with a vertex in $Y_i$ (note that each step involves pairing at $O(\D)$ points). That explains the first term, so we move on to explain the second term. For each point in $e$ that is paired to a point in a vertex $v$ in say $Y_k$, we pair the remaining $k-1$ points in $v$ to determine which edges will be removed from $E$. Usually these $k-1$ points in $v$ will have partners that are in $k-1$ distinct edges (the probability that two would be in the same edge is $O(\D/n)$), so we proceed assuming this is the case. For each edge $e'$ being removed from $E$ we then pair the remaining $r-1$ points of $e'$. So our explanation for the second term is as follows. Each of the $r$ points in $e$ has a probability $\frac{k Y_k }{M_1 } + O\bfrac{\D}{n}$ of being paired to a point in a vertex $v$ in $Y_k$, and then with probability $1-O(\D/n)$ the other $k-1$ points in $v$ all have partners in distinct edges. We now pair the remaining $r-1$ points in each of these $k-1$ edges, and the expected effect of each such pairing on $Y_i$ is $ \frac{(i+1)Y_{i+1}}{M_1} -\frac{iY_i}{M_1} + O\bfrac{\D}{n}$. And so the second line is justified. The third line follows from the definition of $M_2$ and the fourth line is just algebra. 

As is usual in the differential equation method, we will prove that our random variables are concentrated around deterministic counterparts. Letting $t=j/n,$ we will heuristically assume (and later formally show) that $Y_i(j) \approx ny_i(t), M_k(j) \approx nm_k(t)$ for some deterministic functions $y_i(t), m_k(t)$. The expected one-step change in $Y_i$ \eqref{eq:mExpOneStep} gives us differential equations: 
\begin{equation}\label{eqn:ydiffeq}
  y_i' = \rbrac{\frac{(i+1)r(r-1)m_2}{m_1^2}}y_{i+1} - \rbrac{\frac{ir}{m_1} + \frac{ir(r-1)m_2}{m_1^2}}y_{i}  
\end{equation}
and initial conditions
\begin{equation}\label{eqn:ics}
y_i(0) = n_i/n,\quad i=1,\ldots, \D.
\end{equation}
where of course
\begin{equation}\label{eqn:mdef}
   m_0=\sum_{i=1}^{\Delta} y_i,\quad m_1=\sum_{i=1}^\Delta iy_i, \quad m_2 = \sum_{i=1}^\Delta i(i-1)y_i. 
\end{equation}

We will now describe the solution to the above system. It turns out that it can be expressed in terms of the following natural generating function: let
\[
f(x):= \sum_i \frac{n_i}{n} x^i, \qquad g(x):= \frac{1}{d_1^{r-1}} - \int_x^1 \frac{dz}{f'(z)^{r-1}}
\]
and note that $f(0)=0$ and $f$ is strictly increasing on $[0, 1]$, and so $g$ is strictly increasing on $(0, 1]$. In particular $f$ restricted to $[0,1]$ has an inverse. Let
\[
c(t):= f^{-1}(1-rt), \qquad a(t):= f'(c(t))^{r-1} g(c(t))
\]
\begin{claim}
The unique solution to the system \eqref{eqn:ydiffeq} with initial conditions \eqref{eqn:ics} is 
\[
y_i = \frac{a^i}{i!} f^{(i)}(c-a),\quad i=1, \ldots \D
\]
where $f^{(i)}$ denotes the $i^{th}$ derivative of $f$, and functions with suppressed input are evaluated at $t$ (e.g. $c-a$ means $c(t) - a(t)$). 
\end{claim}

\begin{proof}
We first check initial conditions. Note that $c(0) = f^{-1}(1) = 1$. Therefore \newline $a(0) = f'(1)^{r-1} g(1) = d_1^{r-1} \cdot \frac{1}{d_1^{r-1}} = 1$ as well. Thus, evaluating the proposed solution $y_i$ at $t=0$ gives 
\begin{equation*}
  y_i(0) = \frac{1}{i!} f^{(i)}(0) = \frac{n_i}{n}  
\end{equation*}
by the definition of $f$. Now we check the differential equation \eqref{eqn:ydiffeq}. First we establish some identities. If the $y_i$ are the proposed solution, then we have
\begin{equation*}
  m_1 = \sum_i i y_i = a \sum_i \frac{a^{i-1}}{(i-1)!} f^{(i)}(c-a) = af'(c)  
\end{equation*}
where the last inequality follows by noticing the Taylor series for $f'(x)$ centered at $c-a$. Similarly we have $m_2 = a^2 f''(c).$ Also note that 
\begin{equation}\label{eqn:cprime}
      c' = -\frac{r}{f'(c)}
\end{equation}
and since $g'(x) = \frac{1}{f'(x)^{r-1}}$ we have
\begin{align}
  a' &= (r-1)f'(c)^{r-2} f''(c) \cdot \rbrac{-\frac{r}{f'(c)}}g(c) + f'(c)^{r-1} \cdot \rbrac{\frac{1}{f'(c)^{r-1}}}\cdot \rbrac{-\frac{r}{f'(c)}}\nonumber\\
  &= -\frac{r}{f'(c)} - \frac{r(r-1)f''(c)}{f'(c)^2}\cdot a \label{eqn:aprime}
\end{align}
and so 
\begin{equation}\label{eqn:bprime}
  c' -a'= \frac{r(r-1)f''(c)}{f'(c)^2}\cdot a.
\end{equation}
Now to check \eqref{eqn:ydiffeq} we first take the derivative of the proposed solution:
\begin{equation}\label{eqn:yprime}
    y_i' = \frac{a^{i-1} a'}{(i-1)!} f^{(i)}(c-a) + \frac{a^i}{i!} f^{(i+1)}(c-a)(c'-a')
\end{equation}
and then we plug the proposed solution into the right side of \eqref{eqn:ydiffeq} to get
\begin{equation}\label{eqn:checkyprime}
    \frac{(i+1)r(r-1)f''(c)}{f'(c)^2} \cdot \frac{a^{i+1}}{(i+1)!} f^{(i+1)}(c-a) - \rbrac{\frac{ir}{af'(c)} + \frac{ir(r-1)f''(c)}{f'(c)^2}} \frac{a^i}{i!}f^{(i)}(c-a).
\end{equation}
The reader can use \eqref{eqn:cprime} and \eqref{eqn:aprime} to verify that \eqref{eqn:checkyprime} equals \eqref{eqn:yprime}.
\end{proof}

\subsubsection{Estimating the stopping point of the matching process}

In this section we estimate the number of steps taken before the process terminates, assuming that the random variables stay close to their trajectories. The process ends when $M_1(i)=0$ and we assume that $M_1(i) \approx m_1(t) n = a(t) f'(c(t)) n$, so we anticipate that the final value of $t$ is the smallest $t$ such that either $a(t)=0$ or $c(t)=0$ (since $f'(x)=0$ implies $x=0$). But $c(t)= f^{-1}(1-rt) = 0 $ implies that $1-rt = f(0) = 0$ so $t=1/r$. If this were the final value of $t$ it would mean we have an almost perfect matching, so we anticipate that the actual final value of $t$ is where $a(t)=0$. Indeed, note that $g(1)=1/d_1^{r-1}$ and $g'(x) = 1/f'(x)^{r-1} > 1/d_1^{r-1}$ for all $x>0$. Thus $g(x)=0$ for some value of $x>0$ (and this implies that $a(t)=0$ when $c(t)=x$ which happens before $c(t)=0$). Since $a(t) = f'(c)^{r-1}g(c)$ we see that $a(t)=0$ occurs for a smaller value of $t$ than $c(t)=0$. More specifically, $a(t)=0$ occurs when $g(c(t))=0$ which is when $t = c^{-1}(g^{-1}(0)) = (1-g^{-1}(0))/r$. Thus we anticipate that the process runs until time $$t_{end}:= (1-f(g^{-1}(0)))/r.$$ If this is the case the number of unsaturated vertices in the end will be about 
\[
n - rt_{end}n = f(g^{-1}(0)) n.
\]
 \subsection{Independent process}
One step in the independent process consists of the following. We choose the vertex $v$ to be inserted in our independent set $I$. Then we repeat the following for each  point $\hat{v}$ corresponding to $v$. 
We pair $\hat{v}$ and then if its partner is a point in an edge $e$ that now has all but one point paired to partners that are in the independent set and the remaining point of $e$ is unpaired, we then  pair the last unpaired point in that edge, whose partner is in a vertex that will now be \tbf{closed} (no longer eligible for the independent set). When a vertex becomes closed we pair all its points and then put all of the edges corresponding to their partners into a set $\mc{D}$ of \tbf{dead} edges (formally we do not delete the dead edges, since we would like to defer the pairing of all their points until it is necessary). Since each dead edge is incident with a closed vertex, the dead edge will never be the reason why another vertex becomes closed. The edges that are not dead are \tbf{live}.

We let $V_i$ be the set of vertices with $i$ unpaired points, and $L_i$ be the set of live edges with $i$ unpaired points. 


Let \[M_0=\sum_{i=1}^{\Delta} V_i,\quad M_1=\sum_{i=1}^\Delta iV_i, \quad M_2 = \sum_{i=2}^\Delta i(i-1)V_i.\]

We assume that we are at a step We claim that all $i=1\ldots, \Delta$, we have
\begin{align}\E\sqbs{V_i(j+1) - V_i(j)\,|\, \tbf{V}(j), \tbf{L}(j)} &= -\frac{V_i}{M_0} - \sum_{k=1}^\Delta\frac{V_k}{M_0} \cdot k\cdot \frac{2L_2}{M_1} \cdot \frac{i V_i}{M_1} +O\rbrac{\frac{\D^4}{n}}\nonumber\\
&= - \frac{V_i}{M_0}\of{1 + \frac{i\cdot 2L_2}{M_1}}+O\rbrac{\frac{\D^4}{n}} \label{eqn:viExpOneStep}.
\end{align}
We now explain the first line above. The first term corresponds to the event that the vertex $v$ chosen to go into $I$ is from $V_i$, which has probability $ \frac{V_i}{M_0}$. The second term represents the expected number of vertices in $V_i$ that become closed: for each $k$ we have a probability of $\frac{V_k}{M_0}$  that $v$ is a vertex in $V_k$, in which case we pair its $k$ points, each of which has a probability  $\frac{2L_2 + O(\D^2)}{M_1+O(\D^2)} = \frac{2L_2}{M_1} + O\rbrac{\frac{\D^2}{n}}$ of being partnered with a point in an edge that would cause a vertex closure (note that each step involves pairing at $O(\D^2)$ points). Each vertex closure has a probability $\frac{i V_i+O(\D^2)}{M_1+O(\D^2)} = \frac{iV_i}{M_1}+ O\rbrac{\frac{\D^2}{n}}$ of being a vertex in $V_i$. The second line above follows from $\sum_k k V_k = M_1$.


Now we claim that for all $i=1,\ldots r$, we have
\begin{align}
&\E\sqbs{L_i(j+1) - L_i(j)\,|\, \tbf{V}(j), \tbf{L}(j)}\nonumber\\
& \qquad= \sum_{k=1}^\Delta \frac{V_k}{M_0}\cdot k \cdot\of{\frac{(i+1)L_{i+1}}{M_1} - \frac{iL_i}{M_1} - \frac{2L_2}{M_1}\cdot\of{\sum_{r=1}^\Delta\frac{rV_r}{M_1}\cdot(r-1)\cdot\frac{iL_i}{M_1}  } } +O\rbrac{\frac{\D^4}{n}}\nonumber\\
& \qquad= \frac{(i+1)L_{i+1}}{M_0} - \frac{iL_i}{M_0}\cdot\of{ 1 + \frac{2L_2}{M_1}\cdot\frac{M_2}{M_1}   } +O\rbrac{\frac{\D^4}{n}}\nonumber\\
& \qquad= \rbrac{\frac{i+1}{M_0}}L_{i+1} - \of{ \frac{i}{M_0} + \frac{2iL_2M_2}{M_0M_1^2}   }L_i +O\rbrac{\frac{\D^4}{n}}.\label{eqn:lExpOneStep}
\end{align}
Indeed, to explain the second line note that for each $k$ we have  a probability $\frac{V_k}{M_0}$ that $v$ will be in $V_k$, meaning the $k$ points in $v$ would immediately be paired. Each point $\hat{v}$ of $v$ has a probability $\frac{(i+1)L_{i+1} +O(\D^2)}{M_1+O(\D^2)} = \frac{(i+1)L_{i+1} }{M_1}+O\rbrac{\frac{\D^2}{n}}$ of being paired to a point in an edge of $L_{i+1}$ which would then become an edge of $L_i$; and similarly $\hat{v}$ has a probability $\frac{iL_{i}+O(\D^2)}{M_1+O(\D^2)}=\frac{iL_{i}}{M_1}+O\rbrac{\frac{\D^2}{n}}$ of being paired to an edge in $L_i$ which then moves to $L_{i-1}$. We may also lose edges in $L_i$ due to vertex closures: each point $\hat{v}$ in $v$ has a probability $\frac{2L_2+O(\D^2)}{M_1+O(\D^2)} = \frac{2L_2}{M_1}+O\rbrac{\frac{\D^2}{n}}$ of triggering a vertex closure. For each vertex closure and $1 \le r \le \D$ there is a probability $\frac{rV_r+O(\D^2)}{M_1+O(\D^2)}=\frac{rV_r}{M_1}+O\rbrac{\frac{\D^2}{n}}$ that the closed vertex is in $V_r$, in which case we pair the other $r-1$ points in the closed vertex, each of which has a probability $\frac{iL_i+O(\D^2)}{M_1+O(\D^2)}=\frac{iL_i}{M_1} +O\rbrac{\frac{\D^2}{n}}$ of being in $L_i$. So the first line is explained. The second line follows from the definitions of $M_0, M_1, M_2$ and the third line is just algebra. 

Again letting $t=j/n,$ we will heuristically assume (and later formally show) that $L_i(j) \approx n\l_i(t), M_k(j) \approx nm_k(t)$ for some deterministic functions $\l_i(t), m_k(t)$. The expected one-step changes given in \eqref{eqn:viExpOneStep} and \eqref{eqn:lExpOneStep} give us differential equations: 
\begin{align}
    v_i' &= - \frac{v_i}{m_0}\of{1 + \frac{i\cdot 2\l_2}{m_1}}\label{eqn:vdiffeq}\\
    \l_i' &  = \rbrac{\frac{i+1}{m_0}}\l_{i+1} - \of{ \frac{i}{m_0} + \frac{2i\l_2m_2}{m_0m_1^2}   }\l_i\label{eqn:ldiffeq}
\end{align} with initial conditions
\begin{equation}\label{eqn:ics2}
v_i(0) = n_i/n, \quad \l_r(0) = d_1/r, \;\;\; \l_1(0)=\cdots=\l_{r-1}(0)=0
\end{equation}
where of course
\[
m_0=\sum_{i=1}^{\Delta} v_i,\quad m_1=\sum_{i=1}^\Delta iv_i, \quad m_2 = \sum_{i=1}^\Delta i(i-1)v_i.
\]

We will now describe the solution to the above system. Again we have attempted to simplify it as much as possible. We again let
\[
f(x):= \sum_i \frac{n_i}{n} x^i
\]
be the scaled degree generating function. Let $\a=\a(t)$ be the unique solution to the separable differential equation
\begin{equation}
   \a' = \frac{f'(1-\a^{r-1})}{d_1 f(1-\a^{r-1})}, \quad \a(0)=0 
\end{equation}
and let $\b=\b(t)$ be the solution to 
\begin{equation}\label{eqn:betadiffeq}
\b' = -\frac{1}{f(1-\a^{r-1})}, \quad \b(0)=1.
\end{equation}
 
\begin{claim}
The unique solution to the system \eqref{eqn:vdiffeq}, \eqref{eqn:ldiffeq} with initial conditions \eqref{eqn:ics2} is 
\begin{align}
    v_i &= \frac{n_i}{n}  (1-\a^{r-1})^i\b \label{eqn:vsol}\\
    \l_i &= \frac{\binom ri}{rd_1^{i-1}}  \a^{r-i} f'(1-\a^{r-1})^i \b^i
\end{align}
where functions with suppressed input are evaluated at $t$. 
\end{claim}

\begin{proof}
The proof follows by direct substitution to lines \eqref{eqn:vdiffeq} and \eqref{eqn:ldiffeq}. The LHS of \eqref{eqn:vdiffeq} evaluates to
\begin{align}
      &\frac{n_i}{n}\sbrac{\b'(1-\a^{r-1})^i - i(r-1)\a^{r-2}(1-\a^{r-1})^{i-1}\a'}  \nonumber\\
      & \qquad= \frac{n_i}{n}\sbrac{\of{-\frac{1}{f(1-\a^{r-1})}}(1-\a^{r-1})^i - i(r-1)\a^{r-2}(1-\a^{r-1})^{i-1}\of{\frac{f'(1-\a^{r-1})}{d_1 f(1-\a^{r-1})}}}\label{eqn:vcheck}
\end{align}

Meanwhile, substituting our proposed solution we find that 
\[
m_0 = \b f(1-\a^{r-1}), \quad m_1 = \b (1-\a^{r-1})f'(1-\a^{r-1}), \quad m_2 = \b (1-\a^{r-1})^2f''(1-\a^{r-1})
\]
and so substituting the proposed solution to the RHS of \eqref{eqn:vdiffeq} yields
\begin{equation*}
- \frac{\frac{n_i}{n}  (1-\a^{r-1})^i\b}{\b f(1-\a^{r-1})}\of{1 + \frac{i\cdot 2  \frac{\binom r2}{rd_1^{}}  \a^{r-2} f'(1-\a^{r-1})^2 \b^2 }{\b (1-\a^{r-1})f'(1-\a^{r-1}) }}
\end{equation*}
which after simplification matches line \eqref{eqn:vcheck}.

The LHS of \eqref{eqn:ldiffeq} evaluates to
\begin{align}
      & \frac{\binom ri}{rd_1^{i-1}} \Big[     (r-i)\a^{r-i-1} \a' f'(1-\a^{r-1})^i \b^i + \a^{r-i} i f'(1-\a^{r-1})^{i-1} f''(1-\a^{r-1}) (-\a^{r-2})\a' \b^i\nonumber\\
      & \hspace{10cm} + \a^{r-i} f'(1-\a^{r-1})^i i\b^{i-1}\b' \Big]\nonumber\\
            &= \frac{\binom ri}{rd_1^{i-1}} \Bigg[     (r-i)\a^{r-i-1} \of{\frac{f'(1-\a^{r-1})}{d_1 f(1-\a^{r-1})}}  f'(1-\a^{r-1})^i \b^i  \nonumber\\
      & \hspace{2cm}+ \a^{r-i} i f'(1-\a^{r-1})^{i-1} f''(1-\a^{r-1}) (-\a^{r-2})\of{\frac{f'(1-\a^{r-1})}{d_1 f(1-\a^{r-1})}} \b^i \nonumber\\
      & \hspace{5cm}
     + \a^{r-i} f'(1-\a^{r-1})^i i\b^{i-1}\of{ -\frac{1}{f(1-\a^{r-1})}} \Bigg]\nonumber\\
        &= \frac{\binom ri}{rd_1^{i-1}} \Bigg[      \frac{(r-i)\a^{r-i-1}\b^if'(1-\a^{r-1})^{i+1}}{d_1 f(1-\a^{r-1})}   -  \frac{i\a^{2r-i-2}\b^i  f'(1-\a^{r-1})^{i} f''(1-\a^{r-1})}{d_1 f(1-\a^{r-1})}  \nonumber\\
      & \hspace{5cm}
     - \frac{i\a^{r-i}\b^{i-1} f'(1-\a^{r-1})^i }{f(1-\a^{r-1})} \Bigg]\label{eqn:lcheck}
\end{align}

Meanwhile, substituting the proposed solution to the RHS of \eqref{eqn:ldiffeq} yields
\begin{align*}
&\rbrac{\frac{i+1}{\b f(1-\a^{r-1})}}\of{\frac{\binom r{i+1}}{rd_1^{i}}  \a^{r-i-1} f'(1-\a^{r-1})^{i+1} \b^{i+1}} \\
&- \of{ \frac{i}{\b f(1-\a^{r-1})} + \frac{2i \of{  \frac{\binom r2}{rd_1^{}}  \a^{r-2} f'(1-\a^{r-1})^2 \b^2} \cdot \b (1-\a^{r-1})^2f''(1-\a^{r-1})}{\b f(1-\a^{r-1}) \cdot \big(\b (1-\a^{r-1})f'(1-\a^{r-1})\big)^2}   }\frac{\binom ri}{rd_1^{i-1}}  \a^{r-i} f'(1-\a^{r-1})^i \b^i
\end{align*}
which after simplification matches line \eqref{eqn:lcheck}.
\end{proof}

\subsubsection{Estimating the stopping point of the independent process}

The process ends when $M_0(i)=0$ and we know that $M_0(i) \approx m_0(t) n = \b(t) f(1-\a(t)^{r-1})n$, so we anticipate that the final value of $t$ is the smallest $t$ such that either $\b(t)=\g(\a(t))=0$ or $\a(t)=1$ (since $f(0)=0$). So the question is which of $\b^{-1}(0)=\a^{-1}(\g^{-1}(0))$ and $\a^{-1}(1)$ is smaller. Note that since $f'(1) = d_1$ and $f'$ is strictly increasing, by \eqref{eqn:gammadef} we have that $\g(1)<0$. Since $\g$ is decreasing this means that $\g^{-1}(0) < 1$. Since $\a$ is increasing we have that $\a^{-1}$ is also increasing and so $\a^{-1}(\g^{-1}(0)) < \a^{-1}(1)$. Thus we anticipate that the process runs until time $$t_{end}:=\b^{-1}(0) $$ and outputs an independent set of size roughly $\b^{-1}(0)n$.

\section{Dynamic concentration and final proofs}\label{sec:conc}

\begin{theorem}[Warnke]\label{thm:Lutz} Let $\nu, n>1$ be integers. Let $\mc{D} \subseteq \mathbb{R}^{\nu+1}$ be a connected and bounded open set and let $\mc{D}^+$ be the set of points in $\mc{D}$ with only nonnegative coordinates. Let $(F_i)_{1 \le i \le \nu}$ be functions with $F_i:\mc{D}^+ \rightarrow \mathbb{R}$. Let  $\mc{F}_0 \subseteq \mc{F}_1 \subseteq \ldots$ be $\sigma$-fields. Suppose that the random variables $((Y_i(j))_{1 \le i \le \nu}$ are nonnegative and $\mc{F}_j$-measurable for $j>0$. Furthermore, assume that, for all $j>0$ and $1 \le i \le \nu$, the following conditions hold whenever $(j/n, Y_1(j)/n, ..., Y_\nu(j)/n)\in \mc{D}^+$ 
\begin{enumerate}[(i)]
    \item\label{trendLipschitz} $|\Mean{Y_i(j+ 1)-Y_i(j)| \mc{F}_j}-F_i(j/n, Y_1(j)/n, ..., Y_\nu(j)/n)|\le \delta$, where the function $F_i$ is $L$-Lipschitz-continuous on $\mc{D}^+$ (the `Trend hypothesis' and `Lipschitz hypothesis'),
    
    \item $|Y_i(j+1)-Y_i(j)| \le \theta$ (the `Boundedness hypothesis'), and that the following condition holds initially:
    
    \item $\max_{1 \le i \le \nu}|Y_i(0)-\hat{y}_i n| \le \lambda n$ for some $(0,\hat{y}_1,\ldots, \hat{y}_\nu)\in \mc{D}^+$ (the `Initial condition').
\end{enumerate}
Then there are $R=R(\mc{D},(F_i)_{1 \le i \le \nu}, L)\in [1,\infty)$ and $T=T(\mc{D})\in(0,\infty)$ such that, whenever $\lambda > \delta \min\{T, L^{-1}\}+R/n$, with probability at least $1-2 \nu \exp\{-n\lambda^2 /(8T \theta^2)\}$ we have 
\[
\max_{0 \le j \le \sigma n} \max_{1 \le i \le \nu}|Y_i(j)-y_i(j/n)n|<3\exp\{LT\}\lambda n
\]
where $(y_i(t))_{1 \le i \le \nu}$ is the unique solution to the system of differential equations $y'_i(t) =F_i(t, y_1(t), . . . , y_\nu(t))$ with $y_i(0) = \hat{y}_i$ for $1 \le i \le \nu$, and $\sigma=\sigma(\hat{y}_1, . . .\hat{y}_a)\in[0, T]$ is any choice of $\sigma>0$ with the property that $(t, y_1(t), . . . y_\nu(t))$ has $\l^\infty$-distance at least $3\exp\{LT\}\lambda$ from the boundary of $\mc{D}$ for all $t\in[0, \sigma)$.
\end{theorem}

\begin{remark}
The above is a very minor adaptation of Warnke's theorem in \cite{Lutz}. Here we assume that the random variables $Y_i(j)$ are nonnegative, which \cite{Lutz} does not. Also \cite{Lutz} does not have $\mc{D}^+$ (all instances of ``$\mc{D}^+$" are originally ``$\mc{D}$" in \cite{Lutz}). It is not difficult to see that the proof in \cite{Lutz} works for the present adaptation. We use this adaptation because we would like to track random variables that might be 0, meaning that our set $\mc{D}$ must include points with coordinates equal to 0, and since $\mc{D}$ is open it must then also include points with negative coordinates. But we would rather not (for example) consider what the one-step change in a $Y_i$ variable should be if it is negative. 
\end{remark}
\subsection{Proof of Theorem \ref{thm:match}}

\begin{proof}[Proof of Theorem \ref{thm:match}]
We apply Theorem \ref{thm:Lutz}. For this application we have $\nu=\D$ and by equation \eqref{eq:mExpOneStep}
\[F_i(t, y_1, \ldots, y_\D) = \rbrac{\frac{(i+1)r(r-1)m_2}{m_1^2}}y_{i+1} - \rbrac{\frac{ir}{m_1} + \frac{ir(r-1)m_2}{m_1^2}}y_{i}\]
where $m_1, m_2$ are still as defined in \eqref{eqn:mdef} and $y_{\D+1}$ is always 0. For our initial condition we have $\hat{y}_i = n_i/n$. We will let the region $\mc{D}=\mc{D}(\eps)$ be 
\[ \mc{D}:= \left\{(t, y_1, \ldots y_\D) \in \mathbb{R}^{\D+1}: \;\;\;\;-\eps < t < \frac 2r ,\;\;\;\;  y_i > - \eps , \;\;\;\; \eps < m_1 < 2 \sum_{i=1}^\D i n_i/n \right\}\]
which is clearly open, bounded and connected, and for $\eps$ small enough contains the point $(0, n_1/n, \ldots n_\D/n)$ (where we have $m_1 = \sum_{j = 1}^{\D}i n_i/n$). 
Note that (see condition \eqref{trendLipschitz} in Theorem \ref{thm:Lutz} and equation \eqref{eq:mExpOneStep}) we may use $\d = O(\D^2 / n) = o(n^{-1/3})$.
Since $F_i$ can be written as a rational function whose denominator is $m_1^2$, we can take $R = O(1/m_1^2) = O(1)$. Also note that $F_i$ is $L$-Lipschitz continuous for $L= O(1/m_1^3 ) = O(1)$ on $\mc{D}^+$. Since the number of vertices incident with edges revealed in any single step is at most $2\D-1$, we can use $\theta=O(\D)$. We use $T = 2/r = O(1)$. Thus, Theorem \ref{thm:Lutz} gives us a failure probability of at most 
\[
2 \nu \exp\{-n\lambda^2 /(8T \theta^2)\} = 2 \D\exp\{-\Omega\rbrac{n\lambda^2 /\D^2}\} =o(1)
\]
assuming we choose, say $\lambda = n^{-1/10}.$

Furthermore, we can see that the only way the solution to the system of differential equations can ever leave $\mc{D}$ is at a point where $m_1 = \sum_{i = 1}^{\D}i z_i = \eps$. Indeed, since $a$ and $c$ are decreasing and $a(0) = c(0)=1$ we have that $a(t), c(t) \le 1$ for all $t$ and so 
\[m_1(t) = a(t)f'(c(t)) \le 1 \cdot f'(1) = \sum_{i=1}^\D i n_i/n. \]
Also, clearly $y_i \ge 0$ for all $i$. Thus, the only inequality defining $\mc{D}$ that our solution can ever fail to satisfy is $\eps \le rm_1$.
We apply Theorem \ref{thm:Lutz}, and conclude that our discrete random variables $Y_i(j)$ are well approximated by their continuous counterparts $ny_i(j/n)$, for all values of $j \le t_{\eps} n$, where  $t_\eps$ is the value of $t$ such that $m_1(t) = \eps$ (i.e. for values of $j$ corresponding to points in the region $\mc{D} $ defined above). Note that
\begin{align*}
m_1' &= a'f'(c) + af''(c)c' \le -r
\end{align*}
by equations \eqref{eqn:cprime} and \eqref{eqn:aprime}. 
  Since $m_1(t_{end})=0$ we have 
\[t_{end} -  \frac{\eps}{r}\le t_\eps \le t_{end} .\]
We use this to bound the final size of the matching. If we run the process to step $j_{\eps}:=t_{\eps} \cdot n$ then by Theorem \ref{thm:Lutz} whp we arrive at some configuration with $rM(j_{\eps}) = \eps n (1+o(1))$ many points. At this point our matching already has  $j_{\eps} = t_\eps n \ge (t_{end} - \eps)n$ many edges. Also, even if every edge remaining is added to our matching, the final matching will have only $\eps n (1+o(1))$ more edges. Thus the maximum possible number of edges is $\of{t_{end} + \eps+o(1)}n \le (t_{end} + 2\eps)n$.
Altogether the final matching w.h.p. has between $(t_{end} - \eps)n$ and $(t_{end} + 2\eps)n$ many edges. Since $\eps>0$ is arbitrary we are done. 
\end{proof}

\subsection{Proof of Theorem \ref{thm:ind}}

\begin{proof}[Proof of Theorem \ref{thm:ind}]
We apply Theorem \ref{thm:Lutz}. Since we have two types of variables ($V_i$ and $L_i$), when we check condition \eqref{trendLipschitz} we will use functions $F_{v_i}(t, v_1, \ldots, v_\D, \l_1, \ldots, \l_r)$ and $F_{\l_i}(t, v_1, \ldots, v_\D, \l_1, \ldots, \l_r)$ to represent the approximate one-step change in $V_i$ and $L_i$, respectively. For this application we have by equation \eqref{eqn:viExpOneStep}
\[F_{v_i}(t, v_1, \ldots, v_\D, \l_1, \ldots, \l_r) = - \frac{v_i}{m_0}\of{1 + \frac{i\cdot 2\l_2}{m_1}}\]
where $m_0, m_1, m_2$ are still as defined in \eqref{eqn:mdef} and $\l_{r+1}$ is always 0. By equation \eqref{eqn:lExpOneStep} we have
\[F_{\l_i}(t, v_1, \ldots, v_\D, \l_1, \ldots, \l_r) = \rbrac{\frac{i+1}{m_0}}\l_{i+1} - \of{ \frac{i}{m_0} + \frac{2i\l_2m_2}{m_0m_1^2}   }\l_i.\]
 For our initial condition we have by \eqref{eqn:ics2} that $\hat{v}_i = n_i/n,\; \hat{\l}_r = d_1/r$, and  $\hat{\l}_1=\ldots=\hat{\l}_{r-1}=0$. We will let the region $\mc{D}=\mc{D}(\eps)$ be 
\[\left\{(t, v_1, \ldots, v_\D, \l_1, \ldots, \l_r) \in \mathbb{R}^{\D+r+1}: \;\;-\eps < t <  2 ,\;\;  v_i > - \eps , \;\; -\eps < \l_i < \frac{d_1}{r}, \;\;\;\; \eps < m_0 < 2  \right\}\]
which is clearly open, bounded and connected, and for $\eps<1$ contains the point \newline $(0, n_1/n, \ldots n_\D/n, 0, \ldots, d_1/r)$ (where we have $m_0 = 1$). 

Note that (see condition \eqref{trendLipschitz} in Theorem \ref{thm:Lutz} and s \eqref{eqn:viExpOneStep}, \eqref{eqn:lExpOneStep}) we may use $\d = O(\D^4/n)$. For the Lipschitz condition, note that $m_0 \le m_1 \le \D m_0$ and that $F_{v_i}, F_{\l_i}$ can be written as rational functions whose denominator is $m_0m_1^2\ge \eps^3$ on $\mc{D}$. Thus $F_i$ is $L$-Lipschitz continuous (for some $L=L(\eps) = O(1)$) on $\mc{D}$. We have $|V_i(j+1) - V_i(j)| = O(\D)$ since at most one vertex is chosen for the independent set and at most $\D$ become closed in one step. Similarly we have $|L_i(j+1) - L_i(j)| = O(\D^2)$, and so we will use $\theta = O(\D^2)$. 
Thus, Theorem \ref{thm:Lutz} gives us a failure probability of at most 
\[
2 \nu \exp\{-n\lambda^2 /(8T \theta^2)\} = O( \D)\exp\{-\Omega\rbrac{n\lambda^2 /\D^2}\} =o(1)
\]
assuming we choose, say $\lambda = 1/ \log n.$

Furthermore, we can see that the only way the solution to the system of differential equations can ever leave $\mc{D}$ is at a point where $m_0 = \sum_{i = 1}^{\D} v_i = \eps$. Indeed, the $v_i$ are always initially nonnegative and $v_\D$ is initially positive. By equation \eqref{eqn:vdiffeq}, each $v_i$ is nonincreasing as long as they are all nonnegative and $v_\D$ is positive. Let $t_{\eps}$ be the smallest value of $t$ such that $m_0(t)=\eps$.

Note that when $m_0  = \b f(1-\a^{r-1}) = \eps$, since $\a$ is increasing and $f$ is increasing, that $f(1-\a^{r-1}) \ge \k:= f(1-\g^{-1}(0)^{r-1})$ and $\k$ is positive and depends only on the initial degree distribution $\tbf{n}$. Therefore $\b \le \eps / \k$ and since $\b = \g \circ \a$ is decreasing, we have $t_{\eps} \ge \b^{-1} \rbrac{ \frac{\eps}{\k}}$. Now since 
\[
\b' = -\frac{1}{f(1-\a^{r-1})} \in \left[ -\frac{1}{\k}, 0 \right)
\]
we have $(\b^{-1})'  \in \left[-\k, 0 \right)$ and so 
\[
t_{end} \ge t_{\eps} \ge \b^{-1}\rbrac{\frac \eps \k} \ge \b^{-1}(0)   - \frac{\eps}{\k} \cdot \k = t_{end} - \eps.
\]
Thus w.h.p. the process lasts at least until step $j_\eps := t_\eps n$, at which point there are $(1+o(1))\eps n$ open vertices. At that point our independent set already has size $t_\eps n \ge (t_{end} - \eps)n$, and even if all open vertices were ultimately added to the independent set, its final size would be at most $(t_\eps + 2\eps)n \le (t_{end}+ \eps)n$. Thus w.h.p.  the final independent set has size between $(t_{end} - \eps) n$ and $(t_{end} + \eps) n$ and we are done. 
\end{proof}

\section{Upper bounds for the regular case}\label{sec:UBs}

In this section we provide some upper bounds for the likely size of the largest matching and largest independent set in a random regular hypergraph. These bounds follow from the first moment method.

\subsection{Matchings}

Let $X$ be the number of matchings of size $c n$ in the hypergraph $H \sim \mc{H}(n, r, \D)$. Then we have

\begin{align}
    \E[X] &=  \binom{\frac{\D n}{r}}{cn} \frac{(n)_{rcn}\D^{rcn} (\D n - r c n)!}{(\D n)!}. \label{eqn:exmatch}
\end{align}
Indeed, $\binom{\frac{\D n}{r}}{cn}$ is the number of ways to choose our $cn$ edges, and the probability that set of edges forms a matching is calculated as follows: the $r c n$ points in our matching edges must choose distinct vertices to pair with (this can be done in $(n)_{rcn}$ ways); then in each of the aforementioned vertices we must choose one of its $\D$ points to pair with the point from the matching edge, which can be done in $\D^{crn}$ ways; then we pair the rest of the points, and finally divide by the total number of ways we could have paired all the points. 

Now by Stirling's formula we have $x! = \rbrac{ \frac{x}{e}}^x \cdot \exp\{o(x)\}$ as $x\rightarrow \infty$. We let 
\[
h(x) := \left\{\begin{array}{lr}
        x \log x, &  x>0\\
        0, &  x=0.
        \end{array}\right.
\]
Thus \eqref{eqn:exmatch} becomes
\begin{align}
    & \frac{\rbrac{\frac{\D n}{r}} ! \;n!\;\D^{rcn} \;(\D n - r c n)!}{(cn)! \;\rbrac{\frac{\D n}{r} - cn} !\; (n-rcn)! \;(\D n)!} = \frac{\rbrac{\frac{\D n}{er}}^{\frac{\D n}{r}}  \rbrac{\frac{n}{e}}^n \D^{rcn}  \rbrac{\frac{\D n - r c n}{e}}^{\D n - r c n}}{ \rbrac{\frac{c n}{e}}^{cn}  \rbrac{\frac{\frac{\D n}{r} - cn}{e}}^{\frac{\D n}{r} - cn}   \rbrac{\frac{n-rcn}{e}}^{n-rcn}  \rbrac{\frac{\D n}{e}}^{\D n}   }  \exp\{o(n)\} \nn\\
     = & \exp\cbrac{ \sbrac{h\rbrac{\frac{\D}{r}}  + rc \log \D + h(\D  - rc) - h(c) - h\rbrac{\frac{\D }{r} - c}  - h(1-rc) - h(\D )}n +o(n)}
\end{align}
To see the second line,  note that we can cancel a large power of $n/e$. Thus $\E[X]$ goes to 0 so long as we choose $c$ such that 
\[
h\rbrac{\frac{\D}{r}}  + rc \log \D + h(\D  - rc) - h(c) - h\rbrac{\frac{\D }{r} - c}  - h(1-rc) - h(\D ) <0.
\]
At this point we must resort to numerical methods. For a few values of $r, \D$ we provide in Table \ref{tbl2} some upper bounds on the maximum matching size in $\mc{H}(n, r, \D)$ that were verified using the above inequality in Maple. 

The tables below give bounds (which hold w.h.p.) on the maximum matching in $\mc{H}(n, r, \D)$ for a few small values of $r, \D$. By Corollary \ref{cor:matchgirth} these bounds also hold for hypergraphs of high girth. In Table \ref{tbl1}, the $(r, \D)$ entry $b$ indicates that for any sequence of $r$-uniform $\D$-regular hypergraphs $\mc{H}_n$ with $n$ vertices, $\mc{H}_n$ has a matching that leaves at most $bn+o(n)$ vertices unmatched, given by the matching process.

\begin{table}[h!]
\centering
\begin{tabular}{ |c|c|c|c|c| } 
 \hline
  & $\D=2$ & $\D=3$ & $\D=4$ & $\D=5$ \\ 
 \hline
 $r=3$ & 0.250 & 0.250 & 0.239 & 0.227 \\ 
 \hline
 $r=4$ & 0.334 & 0.342 & 0.334 & 0.324\\ 
 \hline
 $r=5$ & 0.397 & 0.411 & 0.406 & 0.397 \\ 
 \hline
\end{tabular}
\caption{Upper bounds}\label{tbl1}
\end{table}
In Table \ref{tbl2} the $(r, \D)$ entry $b$ indicates that w.h.p. any matching in $\mc{H}(n, r, \D)$ must leave at least $bn+o(n)$ vertices unmatched. Note that this implies there exist $r$-uniform $\D$-regular hypergraphs of arbitrarily high girth such that every matching leaves $bn+o(n)$ vertices unmatched.

\begin{table}[h!]
\centering
\begin{tabular}{ |c|c|c|c|c| } 
 \hline
  & $\D=2$ & $\D=3$ & $\D=4$ & $\D=5$ \\ 
 \hline
 $r=3$ & 0.081 & 0.052 & 0.029 & 0.012 \\ 
 \hline
 $r=4$ & 0.158 & 0.138 & 0.116 & 0.096\\ 
 \hline
 $r=5$ & 0.222 & 0.211 & 0.192 & 0.174\\ 
 \hline
\end{tabular}
 \caption{Lower bounds}\label{tbl2}
\end{table}

\subsection{Independent sets}

First we discuss the possibility that $\mc{H}(n, r, \D)$ has an independent set matching a certain trivial upper bound. In particular, note that for $\D \ge 1$, any independent set in a $r$-uniform $\D$-regular hypergraph has size at most $\frac{r-1}{r}n$. Indeed, the complement of an independent set is a vertex-covering of the edges. Since each vertex covers at most $\D$ edges and there are $n\D/r$ edges to cover, any vertex-cover must have at least $n/r$ vertices. 

Now we will see that for certain $r, \D$ that w.h.p. $\mc{H}(n, r, \D)$ has an independent set of size $\frac{r-1}{r}n$ (we assume $r$ divides $n$ of course). We accomplish this by showing that there is a vertex-cover with $n/r$ vertices. To see this, recall that our random hypergraph model is generated from a random pairing of the points in $A$ to the points in $B$, where $A$ (resp. $B$) is the union of disjoint sets we interpret as edges (resp. vertices). Note that we may swap the roles of $A$ and $B$ to obtain $\mc{H}(n\D/r, \D, r)$ instead of $\mc{H}(n, r, \D)$. Now a vertex cover of size $n/r$ in $\mc{H}(n, r, \D)$ corresponds to a set of $n/r$ edges in $\mc{H}(n\D/r, \D, r)$ covering all the vertices, i.e. a perfect matching.  Thus for fixed $r, \D$, it holds that w.h.p. $\mc{H}(n, r, \D)$ has an independent set of size $\frac{r-1}{r}n$ if and only if it holds that w.h.p. $\mc{H}(n\D/r, \D, r)$ has a perfect matching. Now we appeal to the result of Cooper, Frieze, Molloy and Reed \cite{CFMR96} (see equation \eqref{eqn:CFMR}) to obtain the following corollary:

\begin{corollary}[Corollary to \cite{CFMR96}] For $r \ge 2, \D \ge 3$ we have
\begin{equation}
    \lim_{n\to\infty}\P\sqbs{\a(\cH(n,r,\D)) = \frac{r-1}{r}n } = 
\begin{cases}
1 \quad \textrm{if $\D < \sigma_r$} \\
0 \quad \textrm{if $\D > \sigma_r$}
\end{cases}
\end{equation}
where $\s_r := \frac{\log r }{(r -1)\log\bfrac{r}{r-1}} +1$.
\end{corollary}
Note, for example that w.h.p. $\a(\cH(n,r,\D)) = \frac{r-1}{r}n$ when $\D=2$ and $r \ge 3$, or when $\D=3$ and $r \ge 7$.

Now we address the cases where w.h.p. $\a(\cH(n,r,\D)) < \frac{r-1}{r}n$. We will use the first moment method. Let $Y$ be the number of independent sets of size $c n$ in $ \mc{H}(n, r, \D)$. Then we have

\begin{align}
    \E[Y] &=   \sum_{\substack{s_0 + \ldots + s_{r-1} = \frac{\D n}{r} \\ s_1 + 2s_2 + \ldots + (r-1)s_{r-1} = \D cn}} \binom{n}{cn}  \binom{\frac{\D n}{r}}{s_0, \ldots, s_{r-1}} \sbrac{ \prod_{0 \le j \le r-1} \binom{\D}{j}^{s_j}} \frac{(\D cn)!(\D n - \D cn)!}{(\D n)!}\label{eqn:exind}
\end{align}
Indeed, $\binom{n}{cn}$ is the number of ways to choose our $cn$ vertices, and the probability that set of vertices forms an independent set is calculated as follows. Let $S_j$ be the set of edges containing exactly $j$ vertices from the independent set, and $s_j = |S_j|$. So of course $S_{r}=\emptyset$, the total number of edges is $s_0 + \ldots + s_{r-1} = \frac{\D n}{r}$ and the number of points in the independent set is $s_1 + 2s_2 + \ldots + (r-1)s_{r-1} = \D cn$. We designate our edges in $\binom{\frac{\D n}{r}}{s_0, \ldots, s_{r-1}}$ ways. For each $j$ and each edge in $S_j$ we then choose $j$ of its points to be paired to points in the independent set, accounting for the  $\prod_{0 \le j \le r-1} \binom{\D}{j}^{s_j}$. Once that is done we just have to pair the $\D cn$ points in our independent set to the appropriate points on the other side, and then pair off the other $\D n - \D cn$ vertex points. 

The number of terms in \eqref{eqn:exind} is only polynomial, so the sum will go to 0 if each term is exponentially small. We fix some $s_0, \ldots s_{r-1}$ as described in the sum, and we let $x_j := s_j/n$. Using Stirling's formula, the corresponding term is 
\begin{align}
   & \frac{n!\; \rbrac{\frac{\D n}{r}}! \; (\D cn)! \; (\D n - \D cn)!}{(cn)!\; (n-cn)! \; s_0! \ldots s_{r-1}! \; (\D n)!} \prod_{0 \le j \le r-1} \binom{r}{j}^{s_j}\nn\\
   = & \exp\left\{ \left[ h\rbrac{\frac{\D}{r} } + h(\D c) + h(\D(1-c)) + x_1 \log\binom{r}{1} + \ldots x_{r-1} \log\binom{r}{r-1}  \right. \right. \nn\\
   & \qquad \qquad  - h(c) -h(1-c) - h(x_0) - \ldots - h(x_{r-1}) - h(\D)  \bigg] n + o(n) \bigg\} \label{eqn:indmax}
\end{align}
where we have used Stirling's formula, cancelled a large power of $n/e$. We let 
\[
f(x_0, \ldots, x_{r-1}) : = h\rbrac{\frac{\D}{r}}  + h(\D c) + h(\D(1-c)) - h(c) -h(1-c) - h(\D) -\sum_{0 \le j \le r-1}  x_j \log \frac{x_j}{\binom{r}{j}}
\]
be the coefficient of $n$ in the exponent of \eqref{eqn:indmax}. Let us summarize what we know so far. If for some fixed $\D, r, c$ we have that $f(x_0, \ldots, x_{r-1})<0$ for all $x_0, \ldots x_{r-1}$ satisfying
\begin{equation}\label{eqn:constraint1}
 x_j \ge 0, \qquad   x_0 + \ldots + x_{r-1} = \frac{\D}{r}, \qquad x_1 + 2x_2 + \ldots + (r-1) x_{r-1} = \D c
\end{equation}
then w.h.p. $\alpha(\mc{H}(n, r, \D))< cn$. Thus we are interested in maximizing $f$ subject to \eqref{eqn:constraint1}. 

\begin{claim}\label{clm:indupper}
Fix $r \ge 3$, $\D \ge 2$, and $0 < c < \frac{r-1}{r}.$ The maximum of $f$ subject to  \eqref{eqn:constraint1} is given by 
\[
h\rbrac{\frac{\D}{r}}  + h(\D c) + h(\D(1-c)) - h(c) -h(1-c) - h(\D)  - \frac{\D}{r} \log z_1 - c \D \log z_2
\]
where $z_2$ is the unique positive number such that 
\[
\frac{z_2  \sbrac{\rbrac{z_2 + 1}^{r-1} - z_2^{r-1}}}{\rbrac{z_2 + 1}^r - z_2^r} = c
\]
and 
\[
z_1 = \frac{\D}{r\sbrac{\rbrac{z_2 + 1}^r - z_2^r}}.
\]
\end{claim}

\begin{proof}
First we prove that the maximum cannot occur at a point that has $x_k=0$ for any $k$. Consider an arbitrary point $(x_0, \ldots x_{r-1})$ satisfying \eqref{eqn:constraint1} and with $x_k=0$. Note that we must have at least one positive $x_i$. We consider two cases:
\begin{enumerate}
    \item [(i)] there exist two distinct positive entries $x_\l, x_m >0$
    \item[(ii)] we only have one positive $x_\l$ and $x_j=0$ for all $j \neq \l$.
\end{enumerate}
In case (i), we let $\vec{x} = (x_0, \ldots x_{r-1})$ and let $\vec{y} = (y_0, \ldots, y_{r-1})$ where $ y_k = 1, y_\l = \frac{k-m}{m-\l}, y_m = \frac{\l-k}{m-\l}$ and all other coordinates are 0. Since $y_k$ is positive, for $t>0$ small enough we have that $\vec{x} + t\vec{y}$ satisfies the constraints \eqref{eqn:constraint1}. Now note that for $t>0$ small enough
\begin{align*}
    f(\vec{x} + t\vec{y}) - f(\vec{x})& = x_i \log \frac{x_i}{\binom{r}{i}} + x_j \log \frac{x_j}{\binom{r}{j}} - (x_i +ty_i) \log \frac{x_i+ty_i}{\binom{r}{i}} - (x_j +ty_j) \log \frac{x_j+ty_j}{\binom{r}{j}} - t \log\frac{t}{\binom{r}{k}}\\
    & = O(t) - t \log t >0 
\end{align*}
where the last line follows from Lipschitz continuity of the terms $x_i \log \frac{x_i}{\binom{r}{i}}$ and $x_j \log \frac{x_j}{\binom{r}{j}}$. We conclude that $\vec{x}$ is not optimal. 

In case (ii), the constraints \eqref{eqn:constraint1} give us 
\[
x_\l = \frac{\D}{r}  = \frac{\D c}{\l}
\]
so that $c = \frac{\l}{r}$. Since we assume $0 < c < \frac{r-1}{r}$ we have $\l \neq 0, r-1$. We now argue that $\vec{x}$ is nonoptimal similarly to case (i). Choose $m, k$ such that $0 \le m < \l < k \le r-1$, and  again let $\vec{y} = (y_0, \ldots, y_{r-1})$ where $ y_k = 1, y_\l = \frac{k-m}{m-\l}, y_m = \frac{\l-k}{m-\l}$ and all other coordinates are 0. Now since both $y_k, y_m$ are positive, for $t>0$ small enough we have that $\vec{x} + t\vec{y}$ satisfies the constraints \eqref{eqn:constraint1}. Similarly to case (i) we find that $f(\vec{x} + t\vec{y})>f(\vec{x})$ for small positive $t$ as well. Thus, the maximum cannot occur at  any point with any zero coordinates. 

We use the method of Lagrange multipliers. At any maximum of $f$ satisfying the constraints \eqref{eqn:constraint1} there must exist some $\lambda_1, \lambda_2$ such that 
\begin{align}
   & \log\binom{r}{j} - \log x_j - 1 = \lambda_1 + j\lambda_2
\end{align}
which implies
\begin{equation}\label{eqn:47}
    x_j = \binom{r}{j} e^{-1-\lambda_1 - j\lambda_2} = z_1 \binom{r}{j} z_2^j
\end{equation}

where $z_1 := e^{-1-\lambda_1}$ and $z_2 := e^{-\lambda_2}$. Plugging the above into line \eqref{eqn:constraint1} yields 
\begin{equation}\label{eqn:48}
   \frac{\D}{r} =  \sum_{j=0}^{r-1} x_j = \sum_{j=0}^{r-1} z_1 \binom{r}{j} z_2^j = z_1\sbrac{\rbrac{z_2 + 1}^r - z_2^r}
\end{equation}
and
\begin{equation}\label{eqn:49}
   \D c =  \sum_{j=0}^{r-1} j x_j = \sum_{j=1}^{r-1} z_1 j \binom{r}{j} z_2^j = \sum_{j=1}^{r-1} z_1 r \binom{r-1}{j-1} z_2^j =  r z_1 z_2  \sbrac{\rbrac{z_2 + 1}^{r-1} - z_2^{r-1}}.
\end{equation}
Dividing \eqref{eqn:49} by \eqref{eqn:48} yields that 
\begin{equation}\label{eqn:410}
    \frac{z_2  \sbrac{\rbrac{z_2 + 1}^{r-1} - z_2^{r-1}}}{\rbrac{z_2 + 1}^r - z_2^r} = c
\end{equation}
and we can easily solve for $z_1$ in \eqref{eqn:48} to get
\begin{equation}\label{eqn:411}
    z_1 = \frac{\D}{r\sbrac{\rbrac{z_2 + 1}^r - z_2^r}}.
\end{equation}
Finally, note that for these $x_j$ (from \eqref{eqn:47}) we have
\begin{align*}
   & -\sum_{0 \le j \le r-1}  x_j \log \frac{x_j}{\binom{r}{j}}\\
    &= -\sum_{0 \le j \le r-1} z_1 \binom{r}{j} z_2^j \log \rbrac{z_1  z_2^j}\\
    &= -\sum_{0 \le j \le r-1} z_1 \binom{r}{j} z_2^j \rbrac{ \log z_1 + j \log z_2} \\
    & = -z_1 \log z_1 \sum_{0 \le j \le r-1} \binom{r}{j} z_2^j - r z_1 z_2 \log z_2 \sum_{0 \le j \le r-1} \binom{r-1}{j-1} z_2^{j-1}\\
    &= -z_1 \log z_1 \sbrac{(1+z_2)^r - z_2^r}- r z_1 z_2 \log z_2\sbrac{(1+z_2)^{r-1} - z_2^{r-1}}\\
    &= -h(z_1) \sbrac{(1+z_2)^r - z_2^r}- r z_1 h(z_2)\sbrac{(1+z_2)^{r-1} - z_2^{r-1}}\\
    &= - \frac{\D}{r} \log z_1 - c \D \log z_2
\end{align*}
where on the last line we have used \eqref{eqn:410} and \eqref{eqn:411}.
\end{proof}

The tables below give bounds (which hold w.h.p.) on the maximum independent set in $\mc{H}(n, r, \D)$ for a few small values of $r, \D$. In Table \ref{tbl3} the $(r, \D)$ entry $b$ indicates that w.h.p. $\a(\mc{H}(n, r, \D)) \ge bn +o(n)$ (and an independent set of that size is given by the independent process). By Nie and Verstra\"ete \cite{NV} these bounds also extend to high-girth hypergraphs.

\begin{table}[H]
\centering
\begin{tabular}{ |c|c|c|c|c| } 
 \hline
  & $\D=2$ & $\D=3$ & $\D=4$ & $\D=5$ \\ 
 \hline
 $r=3$ & 0.614 & 0.567 & 0.531 & 0.503 \\ 
 \hline
 $r=4$ & 0.708 & 0.670 & 0.640 & 0.616\\ 
 \hline
 $r=5$ & 0.765 & 0.733 & 0.708 & 0.688 \\ 
 \hline
\end{tabular}
 \caption{Lower bounds}\label{tbl3}
\end{table}

In Table \ref{tbl4} the $(r, \D)$ entry $b$ indicates that w.h.p. $\a(\mc{H}(n, r, \D)) \le bn +o(n)$. The $\D=2$ column is the trivial bound $(r-1)/r$ (and we know these bounds are tight), while the other columns are using Claim \ref{clm:indupper}. Note that this implies the existence of $r$-uniform $\D$-regular hypergraphs of high girth whose independence number satisfies the same upper bound. 

\begin{table}[H]
\centering
\begin{tabular}{ |c|c|c|c|c| } 
 \hline
  & $\D=2$ & $\D=3$ & $\D=4$ & $\D=5$ \\ 
 \hline
 $r=3$ & 0.667 & 0.651 & 0.624 & 0.600 \\ 
 \hline
 $r=4$ & 0.750 & 0.744 & 0.724 & 0.706\\ 
 \hline
 $r=5$ & 0.800 & 0.798 & 0.784 & 0.769\\ 
 \hline
\end{tabular}
 \caption{ Upper bounds}\label{tbl4}
\end{table}

\section{Concluding remarks and Open questions}\label{sec:rmk}

The most compelling open questions at the moment are about regular hypergraphs of high girth. Tables \ref{tbl1} through \ref{tbl4} provide bounds on the matching number and independence number. 
\begin{opq}
How much can the entries in Tables \ref{tbl1} through \ref{tbl4} be improved?
\end{opq}
Tables \ref{tbl1} and \ref{tbl3} can most likely be improved by analyzing algorithms slightly better than the greedy algorithms. Tables \ref{tbl2} and \ref{tbl4} might also be improved by a more sophisticated first-moment argument (see for example McKay \cite{MK}). It may also be of interest to try to construct $r$-uniform $\D$-regular hypergraphs of large girth whose independence number (or matching number) is even smaller than $\mc{H}(n, r, \D)$. More precisely consider the following:
\begin{opq}
Fix some $r \ge 3$. Does there exist some $\D=\D(r)$ large enough such that every sequence $\mc{H}_n$ of $r$-uniform $\D$-regular hypergraphs on $n$ vertices with girth tending to infinity has a matching that covers all but at most $o(n)$ vertices? 
\end{opq}
Recall that Cooper, Frieze, Molloy and Reed \cite{CFMR96} proved that $\mc{H}(n, r, \D)$ w.h.p. has a perfect matching for all $\D$ large enough with respect to $r$. Note that for $r=2$ it is known that the answer to the above question is positive (for any $\D \ge 1$). Indeed, Bohman and Frieze \cite{BF11} showed that the Karp-Sipser algorithm produces a matching covering all but $o(n)$ vertices in random regular graphs. Since Karp-Sipser is a local algorithm, the results of Hoppen and Wormald \cite{HopWorm} apply, meaning that Karp-Sipser actually performs approximately the same on all regular graphs of high girth (for an alternative, non-algorithmic, approach to the same problem, see Flaxman and Hoory \cite{FH}). So for $r=2$ the above open question can actually be settled with a local algorithm; thus it seems worthwhile to try the same approach for $r \ge 3$. 

We finish with a question about our main theorems: 
\begin{opq}
Which conditions in Theorems \ref{thm:match} and \ref{thm:ind} can be weakened? What are the weakest assumptions necessary?
\end{opq}

\section{Acknowledgement} The authors would like to thank Lutz Warnke for a helpful conversation about Theorem \ref{thm:Lutz}.

\bibliographystyle{abbrv}
\bibliography{greedy-matching-hg}

\end{document}